\documentclass{article}
\usepackage[utf8]{inputenc}

\usepackage{amsmath}
\usepackage{amsfonts}
\usepackage{amssymb}
\usepackage{amsthm} % place after to make qedhere work
\usepackage{color}
\usepackage{enumerate}
\usepackage{framed}
\usepackage{mathtools}
\usepackage{tcolorbox}
\usepackage{geometry}
\usepackage{graphicx}
\usepackage{paralist} % allows compact enums
\usepackage{mathrsfs}
\usepackage{diagbox}
\usepackage{makecell}
\usepackage{tikz}
\usetikzlibrary{positioning,chains,fit,shapes,calc}
\newtheorem*{thm*}{Theorem}

\numberwithin{equation}{section}
\newtheorem{thm}[equation]{Theorem}
\newtheorem{lem}[equation]{Lemma}
\newtheorem{cor}[equation]{Corollary}
\newtheorem{prop}[equation]{Proposition}

\theoremstyle{definition}

\newtheorem{ex}[equation]{Example}	
\newtheorem{remark}[equation]{Remark}

\newcommand{\ZZ}{\mathbb{Z}}
\newcommand{\NN}{\mathbb{N}}

\newcommand{\cosbase}[3]{#1 \to_{#3} #2}
\newcommand{\cosp}{\cosbase{A}{B}{\rho}}
\begin{document}

\title{On Isospectral Integral Circulant Graphs}
\author{Yan X Zhang \\ San Jos\'{e} State University}
\maketitle

\begin{abstract}
Understanding when two non-isomorphic graphs can have the same spectra is a classic problem that is still not completely understood, even for integral circulant graphs. We say that a natural number $N$ satisfies the \emph{integral spectral Ad\`{a}m property (ISAP)} if any two integral circulant graphs of order $N$ with the same spectra must be isomorphic. It seems to be open whether all $N$ satisfy the ISAP; M\"{o}nius and So showed that $N$ satisfies the ISAP if $N = p^k, pq^k,$ or $pqr$. We show that: (a) for any prime factorization structure $N = p_1^{a_1}\cdots p_k^{a_k}$, $N$ satisfies the ISAP for ``most'' values of the $p_i$;  (b) $N=p^2q^n$ satisfy the ISAP if $p,q$ are odd and $(q-1) \nmid (p-1)^2(p+1)$; (c) all $N =p^2q^2$ satisfy the ISAP. 
\end{abstract}

\section{Introduction}

This work is primarily motivated by the following conjecture given by So in \cite{so2006integral}: ``There are exactly $2^{\tau(N) - 1}$ non-isospectral integral circulant graphs of
order $N$, where $\tau(N)$ is the number of divisors of $N$.'' M\"{o}nius and So \cite{monius2023many}'s work proves this conjecture for:
\begin{itemize}
\item $N = p^k$, where $p$ is a prime $p \geq 2$;
\item $N = pq^k$ or $p^2 q$ with primes $q > p \geq 2$ and integer $k \geq 1$;
\item $N = pqr$ with primes $r > q > p \geq 2$.
\end{itemize}
Otherwise, the conjecture seems to be open.

One way of looking at this conjecture is as an attempt to better understand integral circulant graphs. Circulant graphs of order $N$ are defined by their \emph{symbol} $S \subset \ZZ/N\ZZ$, the set of column indices corresponding to nonzero elements of the first row of the graph's adjacency matrix. In an \textbf{integral} circulant graph, the symbol's information can be compressed into the \emph{integral symbol}, which is a subset of $\{d \colon 1 < d < N, d|N\}$; the main idea is that different indices $k$ in the symbol with the same $\gcd(k,N)$ must occur together or not at all. Thus, there are $2^{\tau(N) - 1}$ possible integral symbols for integral circulant graphs. The authors of \cite{monius2023many} show that for these $N$, different integral symbols must have different spectra. Thus, there must be $2^{\tau{N}-1}$ different spectra for these $N$. In this light, we can also think of So's conjecture as a strengthening of a very similar result:
\begin{thm*}[Klin and Kovacs \cite{klin2012automorphism}]
    There are exactly $2^{\tau(N) - 1}$ non-isomorphic integral circulant graphs of order $N$. 
\end{thm*}

Another way of looking at this conjecture is as a variation on the ``Ad\'{a}m property''. As in \cite{mans2002spectral}, we say that a symbol $S \subset \ZZ/N\ZZ$ has the \emph{Ad\'{a}m property} if it is isomorphic to another symbol $T$ if and only if $S$ and $T$ are \emph{proportional} (that is, one can obtain one symbol to another by multiplying by a common element of $(\ZZ/N\ZZ)^*$). It is natural to say that $N \in \NN$ satisfies the Ad\'{a}m property if all symbols $S \subset \ZZ/N\ZZ$ satisfy the \emph{Ad\'{a}m} property. Ad\'{a}m conjectured \cite{adam1981research} that all natural numbers $N$ satisfy the Ad\'{a}m property; that is, all pairs of isomorphic circulant graphs must have proportional symbols. Several classes of counterexamples were found, such as by Elspas and Turner for $n=16$ \cite{elspas1970graphs} or Alspach \cite{alspach1979isomorphism} for broader classes of $N$. However, the conjecture was also shown to be true for many $N$; for example, Muzychuk proved that the conjecture holds for squarefree $N$ \cite{muzychuk1995adam} and double squarefree $N$ \cite{muzychuk1997adam}.

One natural extension of the Ad\'{a}m property is the following: as in \cite{mans2002spectral}, we say that $N$ satisfies the \emph{spectral Ad\'{a}m property} if two symbols $S, T \subset \ZZ/N\ZZ$ have the same spectrum if and only if $S$ and $T$ are proportional. When we specialize to integral circulant graphs, proportionality is equivalent to equality, because multiplication by an element of $(\ZZ/N\ZZ)^*$ fixes the greatest common divisor with $N$. It is then natural to say that $N$ satisfies the \emph{integral spectral Ad\'{a}m property (ISAP)} if two integral symbols have the same spectrum if and only if they are equal. Note that any $N$ that satisfies the ISAP must have $2^{\tau{N}-1}$ different spectra. In this light, we can reinterpret M\"onius and So's work as proving that $N = p^k, pq^k, pqr$ satisfy the ISAP.

After a quick review and some notation in Section~\ref{sec:prelim}, we start our work in Section~\ref{sec:spectral}, where we introduce some structure that will help us visualize and  manipulate the spectra of integral circulant graphs. Our main contribution here is observing the structural simplicity of $G_N(d)$ for $N=p^n$ and then exploiting the multiplicative structure of the well-known $\mu$ and $\phi$ functions. In Section~\ref{sec:NAR}, we turn our goal of understanding the ISAP into looking for the existence of \emph{nontrivial additive relations (NARs)} that exist on products of $\phi$. In Section~\ref{sec:general} we give ``weak but general'' results that apply to many $N$, but with many assumptions on $N$. In Section~\ref{sec:p2qk}, we prove the $N=p^2q^{n_2}$ case under fewer assumptions. In Section~\ref{sec:p2q2} we give our main ``narrow but strong'' result that solves $N=p^2q^2$ completely. We end with some remarks in Section~\ref{sec:conclusion}.

\section{Preliminaries}
\label{sec:prelim}

\subsection{Additive Relations and NARs}

Given a vector $v$ with some index set $S$, we use $v[s]$ to denote the entry corresponding to $s \in S$ in $v$. For matrices $M$, we use $M[s, t]$ to denote the entry in row $s$ and column $t$.

For a vector $v$, we define an \emph{additive relation on $v$} to be a relation of the form 
\[
  \sum_{x \in S_1} v[x] = \sum_{x \in S_2} v[x]
\]
for distinct subsets $S_1, S_2 \subset S$. Sometimes it will be more convenient to rewrite it as a single equation
\[
\sum_i a_i v[i] = 0,
\]
where each $a_i$ is in $\{-1, 0, 1\}$ depending on whether $v[i]$ appears only on the left, on both sides, or only on the right respectively. We call this the \emph{one-sided version} of the additive relation.

Furthermore,
\begin{enumerate}
    \item we say that an additive relation $X$ is \emph{nontrivial} if $S_1 \neq S_2$. We abbreviate a nontrivial additive relation\footnote{In our context it is important to be careful. For example, if our vector is $v = [0, 0, 2]$ indexed by $[1, 2, 3]$, then $v[0] = v[1]$ is a NAR because the indices are different, even if the values are the same.} as \emph{NAR}.
    \item if $S_1$ and $S_2$ are disjoint and nonempty, we call the (necessarily nontrivial) additive relation a \emph{disjoint NAR}. Any NAR $X$ on non-equal $S_1$ and $S_2$ creates a disjoint NAR $d(X)$ if we remove the intersection $S_1 \cap S_2$. We call $d(X)$  the \emph{reduction} of $X$ and say that $X$ \emph{reduces} to $d(X)$.
    \item we say that an additive relation $X$ \emph{involves (an index)} $x$ (equivalently, $x$ \emph{is involved in} $X$) if $x$ appears in the equation as $x \in S_1$ or $x \in S_2$. We similarly say that $X$ \emph{involves (a value)} $y$ if $v[x] = y$ for some index $x$ involved in $X$. We will often just say ``involves'' when the context is clear. Also, we say that an index (or value) is \emph{involved nontrivially} in a NAR $X$ if the corresponding value does not only appear on one side of $X$ (equivalently, the corresponding value is not involved in the reduction $d(X)$. 
\end{enumerate}

In \cite{so2006integral}, So defines a \emph{super sequence} to be a sequence of natural numbers $a_1 < a_2 < \cdots < a_k$ such that for all $s < k$, $a_{s+1} > \sum_{i=1}^s a_s$. It is easy to observe that

\begin{prop}
\label{prop:super-sequence}
There cannot exist NARs on a super sequence.
\end{prop}

\subsection{A Review of Spectral Theory of Integral Circulant Graphs}
\label{sec:review}

The material in this section can be found in \cite{so2006integral} and \cite{monius2023many}. A \emph{circulant graph} $CG_N(S)$ of order $N$ is characterized by a \emph{symbol} $S$, which is a subset of $\{d: 1 \leq d < N\}$ where $i \in S$ if and only if $(N-i) \in S$. The graph is constructed by labeling the vertices $0, \ldots, N-1$ and creating an edge $(i, j)$ if and only if $i-j \in S$. We can then write down the \emph{spectrum} (eigenvalues) of $CG_N(S)$ as the multiset
\[
Sp(CG_N(S)) = \{\lambda_0(S), \lambda_1(S),...,\lambda_{N-1}(S)\}
\]
where for $0 \leq t < N$, \[
\lambda_t(S) = \sum_{j \in S} \omega^{tj}.\]

An \emph{integral circulant graph (ICG)} is a circulant graph where the spectrum consists only of integers. So \cite{so2006integral} showed that integral circulant graphs are characterized by symbols $S$ where all the indices $k$ with the same $\gcd(k, N)$ must appear at the same time or not at all. In other words, we can define $\tau(N) - 1$ \emph{basic integral symbols} $\{G_N(d): d|N, d < N\}$, where \[
G_N(d) = \{k: \gcd(N,k) = d\} \subset [N-1].
\] (we use $[k]$ to denote the set $\{1, 2, \ldots, k\}$) These $\tau(N) - 1$ basic integral symbols partition $[N-1]$. Then there are exactly $2^{\tau(N) - 1}$ integral circulant graphs of order $N$, which corresponds to a choice to include all the values in each basic integral symbol or not. We can then compute the spectrum by just adding the corresponding spectra for the basic integral symbols, because the matrices corresponding to different basic integral symbols pairwise commute.

In other words, we can use $ICG_N(D)$ to denote the integral circulant graph of order $N$ with the \emph{integral symbol} $\cup_{d \in D} G_N(d)$ (formally, $ICG_N(D) = CG_N(\cup_{d \in D} G_N(d))$). Its eigenvalues can then be computed as
\[
\lambda_t(D) = \lambda_t(\cup_{d \in D} G_N(d)) = \sum_{d \in D} \lambda_t(G_N(d)).
\]
The spectra of the basic integral symbols can then be computed with the Euler function $\phi$ and M\"obius function $\mu$ as follows: for $0 \leq t < N$, 
\begin{equation}
\label{eqn:basic}
\lambda_t (G_N(d)) = \frac{\phi(N/d)}{\phi \left(\frac{N/d}{\gcd(t, N/d)} \right)} \mu \left( \frac{N/d}{\gcd(t, N/d)} \right).
\end{equation}

%For a natural number $N$, our goal is often to say ``if $Sp(ICG_{N}(D_A)) = Sp(ICG_{N}(D_B))$ for some sets $A$ and $B$, $D_A = D_B$.'' In this case, we say that $N$ has the \emph{integral Ad\`{a}m property (IAP)}. 

\section{Spectral Theory of Integral Circulant Graphs}
\label{sec:spectral}

\subsection{The Spectral Theory for $N = p^{n}$}

Let $N = p^n$ where $p$ is a prime. Using Equation~\ref{eqn:basic}, we can characterize the spectra of the basic integral symbols (here they must be powers of $p$ as the following): 
\[
  \lambda_t(G_N(p^{n-\beta})) = \frac{\phi(p^\beta)}{\phi \left( \frac{p^\beta}{\gcd(t, p^\beta)} \right)} \mu \left( \frac{p^\beta}{\gcd(t, p^\beta)}\right),
\]
where $t$ indexes over $1 \leq t \leq N$ and  $\beta$ indexes over $0 \leq \beta \leq n$. When $\beta$ is fixed, this assigns to each $G_N(p^{n-\beta})$ a $N=p^n$-dimensional vector $\lambda$ indexed by $t$, which is the spectrum of $G_N(p^{n-\beta})$ when viewed as a multiset.

First, observe that $\gcd(t,p^x)$ only depends on $\gamma$ where $p^\gamma\| t$ (we use $p^k\|t$ to denote that $p^k|t$ and $p^{k+1} \nmid t$). This means it suffices to only consider $t = p^\gamma$. We obtain

\[
  \lambda_{p^\gamma}(G_N(p^{n-\beta})) = \frac{\phi(p^\beta)}{\phi \left( \frac{p^\beta}{p^{\min(\beta, \gamma)}}\right)} \mu\left( \frac{p^\beta}{p^{\min(\beta, \gamma)}}\right).
\]

An equivalent formulation to the above computation is
\begin{prop}
  \label{prop:lambda-values} The values  $ \lambda_{p^\gamma}(G_N(p^{n-\beta}))$ can take are: \begin{enumerate}
  \item $1$ if $\beta = 0$. Otherwise,
  \item $-p^{\beta - 1}$ if $\gamma = \beta + 1$.
  \item $\phi(p^\beta) = (p-1)p^{\beta - 1}$ if $\gamma \geq \beta$.
    \item $0$ if $\beta > \gamma + 1$.
    \end{enumerate}
  There are exactly $\phi(p^{\beta - \gamma})$ different $t$ in $[N-1]$ such that $\gcd(t, p^\beta) = p^\gamma$.  
\end{prop}

To encode this information, we can define an $(n+1) \times (n+1)$ matrix $M(N) = M(p^n)$ with rows labeled by $\gamma \in \{0, 1, \ldots, n\}$ and columns labeled by $\beta \in \{0, 1, \ldots, n \}$, where
\[
M(N)[\gamma, \beta] = \lambda_{p^\gamma}(G_N(p^{n-\beta})).
\]
(we use $M[r, c]$ to denote the entry of $M$ in row $r$ and column $c$)

\begin{tabular}{c|c|c|c|c|c|c|c|c}
  \diagbox[innerwidth=1cm, height=4ex]{$\gamma$}{$\beta$} & $0$ & $1$ & $2$ & $3$ & $4$ & $\cdots$ & $n$ &\\
  \hline
             $0$ & $1$ & $-1$ & $0$ & $0$ & $0$ & $\cdots$ & $0$ & $\times$ $\phi(p^n)$ \\
  \hline
             $1$ & $1$ & $p-1$ & $-p$ & $0$ & $0$ & $\cdots$ & $0$ & $\times \phi(p^{n-1})$ \\
  \hline
             $2$ & $1$ & $p-1$ & $(p-1) p$ & $-p^2$ & $0$ & $\cdots$ & $0$ & $\times$ $\phi(p^{n-2})$ \\
  \hline
             $3$ & $1$ & $p-1$ & $(p-1)p$ & $(p-1)p^2$ & $-p^3$ & $\cdots$ & $0$ & $\times \phi(p^{n-3})$ \\
  \hline
             $4$ & $1$ & $p-1$ & $(p-1)p$ & $(p-1)p^2$ & $(p-1)p^3$ & $\cdots$ & $0$ & $\times \phi(p^{n-4})$ \\
  \hline
  $\vdots$ & $\vdots$ & $\vdots$ & $\vdots$ & $\vdots$ & $\vdots$ & $\vdots$ & $\vdots$ & $\vdots$ \\
  \hline
  $N$ & $1$ & $\phi(p)$ & $\phi(p^2)$ & $\phi(p^3)$ & $\phi(p^4)$ & $\cdots$ & $\phi(p^n)$ & $\times \phi(1) = 1$
\end{tabular}

In $M(N)$, we say that row $i$ has \emph{(row) multiplicity} $\phi(p^{n-i})$, corresponding to the fact that the entries in that row appears $\phi(p^{n-i})$ times in each spectrum vector. 

Formally, let the \emph{extended form} matrix $\overline{M(N)}$ be the $N \times (n+1)$ matrix where the $(n+1)$ columns correspond to the spectra of the $(n+1)$ basic integral symbols for $N$. We call $M(N)$ the \emph{compact form} of $\overline{M(N)}$ since it stores the same information but with only $(n+1)$ rows, with row $i$ appearing $\phi(p^{n-i})$ times in $\overline{M(N)}$. As a sanity check,
\[\phi(1) + \phi(p) + \cdots + \phi(p^n) = p^n = N,\]
so the number of rows works out.

\begin{ex}

As an example, take $N = 2^3$. Then the compact form $M(8)$ equals (with multiplicities on the right):

\begin{center}
\begin{tabular}{c|c|c|c|c|c}
  \diagbox[innerwidth=1cm, height=4ex]{$\gamma$}{$\beta$} & $0$ & $1$ & $2$ & $3$ & \\
  \hline
             $0$ & $1$ & $-1$ & $0$ & $0$ & $\times$ $4$ \\
  \hline
             $1$ & $1$ & $1$ & $-2$ & $0$ &$\times 2$ \\
  \hline
             $2$ & $1$ & $1$ & $2$ & $-4$ & $\times 1$ \\
  \hline
             $3$ & $1$ & $1$ & $2$ & $4$ & $\times 1$ \\
  \hline
  \end{tabular}
\end{center}
We can write the extended form $\overline{M(8)}$ as 
\begin{center}
\begin{tabular}{c|c|c|c|c}
  \diagbox[innerwidth=1cm, height=4ex]{$\gamma$}{$\beta$} & $0$ & $1$ & $2$ & $3$ \\
  \hline
             $0$ & $1$ & $-1$ & $0$ & $0$  \\
  \hline
             $0$ & $1$ & $-1$ & $0$ & $0$  \\
  \hline
             $0$ & $1$ & $-1$ & $0$ & $0$  \\
  \hline
             $0$ & $1$ & $-1$ & $0$ & $0$  \\
  \hline
             $1$ & $1$ & $1$ & $-2$ & $0$  \\
  \hline
             $1$ & $1$ & $1$ & $-2$ & $0$  \\
  \hline
             $2$ & $1$ & $1$ & $2$ & $-4$ \\
  \hline
             $3$ & $1$ & $1$ & $2$ & $4$  \\
  \hline
  \end{tabular}
\end{center}

The columns (except the leftmost) of $\overline{M(8)}$ are the spectra of different $G_N(d)$'s. So each of the $8$ subsets $S$ of the $\tau(8) - 1 = 3$ columns on the right with $\beta > 0$ corresponds to a different ICG $G$; summing the columns of $\overline{M(8)}$ over $S$ produces a vector containing the spectrum of $G$.
\end{ex}

\subsection{Using the Multiplicative Structure of $\phi$ and $\mu$}

Let $N=p^{n_1}q^{n_2}$, $p^{\gamma_1} \| t$, and $q^{\gamma_2} \| t$. Observe that 
\[
\lambda_t(G_N(p^{n_1 - \beta_1}q^{n_2-\beta_2})) = \lambda_{p^{\gamma_1}}(G_{p^{n_1}}(p^{n_1 - \beta_1}))\lambda_{p^{\gamma_2}}(G_{p^{n_2}}(p^{n_2 - \beta_2})).
\]
This is because the only terms that appear in $\lambda$ are $\phi$ and $\mu$, which are multiplicative functions. As the row multiplicities are also just $\phi$ functions, this means we can obtain $M(N)$ by taking $M(p^{n_1})$ and $M(q^{n_2})$ and taking their tensor product! Formally, we define a matrix $M(N)$ whose rows and columns are both labeled by $(m_1, m_2)$ where $m_1 \in \{0, \ldots, n_1\}$ and $m_2 \in \{0, \ldots, n_2\}$, and then construct the entry
\[
M(N)[(m_1, m_2), (k_1, k_2)] = M(p^{n_1})[(m_1, k_1)] M(p^{n_2})[(m_2, k_2)]
\]
where the row $(m_1, m_2)$ in $M(N)$ has multiplicity
\[
\phi(m_1, m_2) = \phi(m_1)\phi(m_2).
\]

\begin{ex}
As an example, with
\[
  M(4) = \begin{bmatrix} 1 & -1 & 0 & (\times 2) \\ 1 & 1 & -2 & (\times 1) \\ 1 & 1 & 2 & (\times 1) \end{bmatrix};   M(9) = \begin{bmatrix} 1 & -1 & 0 & (\times 6) \\ 1 & 2 & -3 & (\times 2) \\ 1 & 2 & 6 & (\times 1) \end{bmatrix},
\]
we can obtain $M(36)$ by tensoring them to obtain a $9 \times 9$ matrix
\[
  \begin{bmatrix}
    1 & -1 & 0 & -1 & 1 & 0 & 0 & 0 & 0 & (\times 12) \\
    1 & 2 & -3 & -1 & -2 & 3 & 0 & 0 & 0  & (\times 4) \\
    1 & 2 & 6 & -1 & -2 & -6 & 0 & 0 & 0  & (\times 2) \\
    1 & -1 & 0 & 1 & 1 & 0 & -2 & 2 & 0  & (\times 6) \\
    1 & 2 & -3 & 1 & 2 & -3 & -2 & -4 & 6  & (\times 2) \\
    1 & 2 & 6 & 1 & 2 & 6 & -2 & -4 & -12  & (\times 1) \\
    1 & -1 & 0 & 1 & -1 & 0 & 2 & -2 & 0  & (\times 6) \\
    1 & 2 & -3 & 1 & 2 & -3 & 2 & 4 & -6  & (\times 2) \\
    1 & 2 & 6 & 1 & 2 & 6 & 2 & 4 & 12  & (\times 1) 
 \end{bmatrix}.
\]
The row multiplicities add up to $36$, as expected.
\end{ex}

To summarize, for $N = p_1^{n_1} p_2^{n_2}\cdots p_r^{n_r}$, we can find an $(n_1 + 1)\cdots(n_r+1) \times (n_1 + 1)\cdots(n_r + 1)$ matrix $M(N)$ with the rows and columns labeled by $(m_1, \ldots, m_r)$, where $m_i \in \{0, \ldots, n_i\}$ for all $i$. We call this common indexing set $I(N) = \{0, 1, \ldots, n_1\} \times \cdots \times \{0, 1, \ldots, n_r\}$.

For each $(m_1, \ldots, m_r)$, the numbers $\phi(p_1^{m_1}\cdots p_r^{m_r})$ appear twice. On each column $(m_1, \ldots, m_r)$, they appear as the entry in the final row $(n_1, \ldots, n_r)$. On each row $(n_1-m_1, \ldots, n_r-m_r)$, they appear as the row's multiplicity in the $N \times (n_1 + 1)\cdots(n_r + 1)$ extended form matrix $\overline{M(N)}$. We define
\[
\phi(m_1, m_2, \ldots, m_r) \coloneqq \phi(p_1^{m_1}\cdots p_r^{m_r})
\]
and
\[
P(N) \coloneqq \{\phi(m_1, \ldots, m_r): (m_1, \ldots, m_r) \in I(N)\},
\]
which we can consider to be a vector indexed by $I(N)$.

\section{Cospectral Pairs and Nontrivial Additive Relations}
\label{sec:NAR}

\subsection{Cospectral Pairs}

For $a = (m_1,\ldots, m_r) \in I(N)$ the index of some column of $M(N)$, let $v_a$ be the corresponding column of $M(N)$ and $\overline{v_a}$ be the corresponding column of $\overline{M(N)}$. For any subset $A \subset I(N)$ of the columns, we use $v_A$ to mean $\sum_{a \in A} v_a$, and similarly $\overline{v_A} = \sum_{a \in A} \overline{v_a}.$

\begin{prop}
    \label{prop:ISAP=NAR} $N$ does not satisfy the ISAP if and only if there exist $2$ different subsets $A$ and $B$ of $I(N) \backslash \{(0, \ldots, 0\}$ and a permutation $\rho \in S_{N}$ such that for all $i \in [N]$, $(\overline{v_A})_i = (\overline{v_B})_{\rho(i)}$. 
\end{prop}
\begin{proof}
This is just a reformulation of Section~\ref{sec:spectral}; the $2^{\tau(n)-1}$ subsets of the columns of $\overline{M(N)}$ (except for the leftmost column with index $(0, \ldots, 0)$) generate the different possible spectra of integral circulant graphs with $N$ vertices by summation. Two vectors represent the same spectra if and only if they equal under some permutation.  
\end{proof}

\begin{remark}
In Proposition~\ref{prop:ISAP=NAR}, the statement holds even if we replace ``$I(N) \backslash \{(0, \ldots, 0\}$'' by ``$I(N)$.'' This is because the leftmost column in $\overline{M(N)}$ has sum $N$ (being the all $1$'s vector) while all other columns have sum $0$. If $\overline{v_A} = \overline{v_B}$ as multisets, the sums of their entries must be equal as well, which means they must either both contain the first column or both fail to contain the first column.
\end{remark}

%This means detecting whether $N$ satisfies the ISAP can be reduced to the following problem:

%\begin{tcolorbox}
%Does there exist $2$ subsets $A$ and $B$ of the columns $I(N) \backslash \{(0, \ldots, 0\}$ such that $\overline{v_A}$ is equal to $\overline{v_B}$ as multisets?  
%\end{tcolorbox}

If such $A$, $B$, $\rho$ exist, we call them a \emph{cospectral} pair denoted  by $\cosp$, and say the $\rho$ \emph{connects} $A$ to $B$. We have therefore reduced the decision problem of finding if $N$ satisfies the ISAP to the existence of cospectral pairs on $I(N)$. We will soon see that this in turn reduces to the existence of certain additive relations on $P(N)$.

\subsection{The Row NARs Induced by a Cospectral Pair}
\label{sec:row-NARs}

Given a cospectral pair $\cosp$, construct a bipartite graph $\overline{G_\rho}$ on vertices $[N] \times \{0, 1\}$, with the vertices $(*, 0)$ on the left and $(*, 1)$ on the right,  such that there exists an edge $((x, 0), (y, 1))$ if and only if $\rho(x) = y$. We then construct a similar graph $G_\rho$ on vertices $I(N) \times \{0, 1\}$ (also with the $(*, 0)$ vertices on the left and the $(*, 1)$ vertices on the right) such that there exists an edge $((x, 0), (y,1))$ if and only if there exists some $x', y' \in [N]$ where row $x'$ (resp. $y'$) in $\overline{M(N)}$ is a copy of row $x$ (resp. $y$) in $M(N)$. If we index $\overline{v_A}$ by the $(*, 0)$ and $\overline{v_B}$ by the $(0, *)$ in $\overline{G_\rho}$, we can see that edges connect (some but not necessarily all) pairs of values of the $\overline{v_*}$ with the same value. The same is true if we look at the compressed vectors $v_A$ and $v_B$. 

Consider a connected component in $G_\rho$. We can write it as $Y_1 \cup Y_2$, where $Y_1$ are the vertices on is on the left and the $Y_2$ are the vertices on the right. This lifts (given an edge $((x,0), (y,1))$ in $G_\rho$, take all the edges $((x', 0), (y', 1))$ where $x'$ are copies of $x$ and $y'$ are copies of $y$ in $\overline{G\rho}$) to some $\overline{Y_1} \cup \overline{Y_2}$ in $\overline{G_\rho}$, which must be a matching (a collection of disjoint edges) because $\overline{G\rho}$ is itself a matching. This means \[
|\overline{Y_1}| = |\overline{Y_2}|.\]
Every $(r, 0)$ in $Y_1$ accounts for $w(r)$ vertices in $\overline{Y_1}$, and similarly for $(r,1)$ in $Y_2$. Therefore, our equality translates to an additive relation $X_1$ of the form
\[
\sum_{(r, 0) \in Y_1} w(r) = \sum_{(r,1) \in Y_2} w(r)
\]
on the row weights $w(r) \in P(N)$. 

Suppose $G_\rho$ had $s$ connected components. Then iterating our process $s$ times creates $s$ additive relations $X_1, \ldots, X_s$ on $P(N)$ such that each element of $P(N)$ appears exactly once on the left and exactly once on the right among the $X_i$. In this case, we say that $\cosp$ \emph{induces} relations $X_1, \ldots, X_s$.

Suppose that for some row $r$ in $M(N)$, all copies $r'$ of $r$ in $\overline{M(N)}$ satisfy $\rho(r')=r'$. Then we say that $\rho$ \emph{fixes} $r$ and call the corresponding trivial relation $w(r) = w(r)$ \emph{fixed}. We call a $\rho$ \emph{simplified} if for all $r$ where $v_A[r] = v_B[r]$, $\rho$ fixes $r$. Then,
\begin{prop}
\label{prop:simplified}
Suppose $A \to_{\rho'} B$ is a cospectral pair. Then there exists a simplified $\rho$ such that:
\begin{enumerate}
    \item $A \to_{\rho} B$ is also a cospectral pair.
    \item $\cosp$ induces $s$ row relations $X_1, \ldots, X_s$, which are all either fixed or disjoint.
\end{enumerate} 
\end{prop}
\begin{proof}
Suppose there is some $r$ such that $v_A[r] = v_B[r]$. Then let $Y_1 \cup Y_2$ be the connected component in $G_{\rho'}$ containing $(r,0)$ and $(r,1)$. We can construct $\rho$ from $\rho'$ by just letting $\rho(x) = x$ for all copies $x$ of $r$, and re-map the other edges arbitrarily in $\overline{Y_1} \cup \overline{Y_2}$ (which does not affect the validity of $\rho'$ as all the vertices involved in this component correspond to the same value in $v_A$ or $v_B$). As a result, we have created a connected component of a single edge in two vertices $\{(r, 0) \cup (r,1)\}$ in $G_\rho'$, corresponding to a fixed relation $w(r) = w(r)$. Repeating, the remaining non-fixed row relations must then be disjoint as none of them can use both $(r,0)$ and $(r,1)$ for any $r$.
\end{proof}

\textbf{From this point on, we always assume $\rho$ is simplified.} We call the resulting disjoint relations the \emph{row NARs induced by $\cosp$.} 

\begin{ex}
    \label{ex:simplification}
We give an example of the simplification process in Figure~\ref{fig:rho}. A possible (compact form) pair of cospectral $v_A$ and $v_B$ connected by some $\rho$ is shown in Equation~\ref{eqn:simplify-example}. 

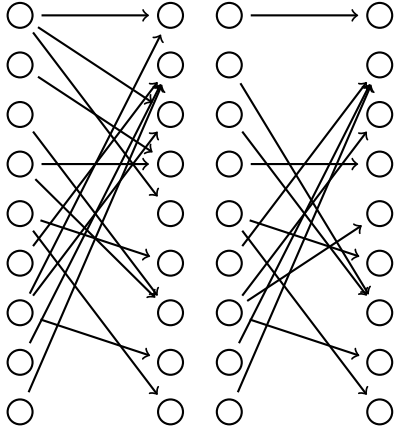
\begin{figure}
    \centering

\begin{tabular}{cc}
\begin{tikzpicture}[thick, every node/.style={draw,circle},
  every fit/.style={ellipse,draw,inner sep=-2pt,text width=2cm},
  ->,shorten >= 3pt,shorten <= 3pt, scale=0.5]
% the vertices of U
\begin{scope}[start chain=going below,node distance=3mm]
\foreach \i in {00,01,02,10,11,12,20,21,22}
  \node[on chain] (f\i) {};
\end{scope}

% the vertices of V
\begin{scope}[xshift=4cm,start chain=going below,node distance=3mm]
\foreach \i in {00,01,02,10,11,12,20,21,22}
  \node[on chain] (s\i) {};
\end{scope}

% the edges
\draw (f00) -- (s00); %0
\draw (f00) -- (s11);
\draw (f00) -- (s02);
\draw (f20) -- (s21);
\draw (f20) -- (s02);
\draw (f20) -- (s00);
\draw (f11) -- (s12); % 2
\draw (f11) -- (s22);
\draw (f01) -- (s10); %1
\draw (f02) -- (s20);
\draw (f10) -- (s20);
\draw (f10) -- (s10);
\draw (f12) -- (s01); %3
\draw (f21) -- (s01);
\draw (f22) -- (s01);
\end{tikzpicture} &

\begin{tikzpicture}[thick, every node/.style={draw,circle},
  every fit/.style={ellipse,draw,inner sep=-2pt,text width=2cm},
  ->,shorten >= 3pt,shorten <= 3pt, scale=0.5]

% the vertices of U
\begin{scope}[start chain=going below,node distance=3mm]
\foreach \i in {00,01,02,10,11,12,20,21,22}
  \node[on chain] (f\i) {};
\end{scope}

% the vertices of V
\begin{scope}[xshift=4cm,start chain=going below,node distance=3mm]
\foreach \i in {00,01,02,10,11,12,20,21,22}
  \node[on chain] (s\i) {};
\end{scope}

% the edges
\draw (f00) -- (s00); %0
\draw (f20) -- (s11);
\draw (f20) -- (s21);
\draw (f20) -- (s02);
\draw (f11) -- (s12); % 2
\draw (f11) -- (s22);
\draw (f01) -- (s20);
\draw (f02) -- (s20);
\draw (f10) -- (s10);
\draw (f12) -- (s01); %3
\draw (f21) -- (s01);
\draw (f22) -- (s01);
\end{tikzpicture}
\end{tabular}
    \caption{Left: A possible $G_\rho$. These edges index the $36$ edges in $\overline{G_\rho}$. Right: after simplification. Any connected component involving two ``matching'' vertices will have had those two vertices isolated into a single component of their own.}
    \label{fig:rho}
\end{figure}

\begin{equation}
    \label{eqn:simplify-example}
\rho \left(
 \begin{bmatrix}
   \textbf{0} & (\times 12) \\
   1  & (\times 4) \\
   1  & (\times 2) \\
   \textbf{1}  & (\times 6) \\
   2  & (\times 2) \\
   3  & (\times 1) \\
   0  & (\times 6) \\
   3  & (\times 2) \\
   3  & (\times 1) 
 \end{bmatrix} \right) = 
  \begin{bmatrix}
   \textbf{0} & (\times 12) \\
   3  & (\times 4) \\
   0  & (\times 2) \\
   \textbf{1}  & (\times 6) \\
   0  & (\times 2) \\
   2  & (\times 1) \\
   1  & (\times 6) \\
   0  & (\times 2) \\
   2  & (\times 1) 
 \end{bmatrix}
\end{equation}

The set of values appearing in these vectors is $\{0, 1, 2, 3 \}$. It is possible to pick $\rho$ such that $G_\rho$ is as in Figure~\ref{fig:rho}. There are $4$ connected components, which induce $4$ additive relations
\begin{align*}
    \textbf{12} + 6 &= \textbf{12} + 2 + 2 + 2 \\
    4 + 2 + \textbf{6} &=  6 + \textbf{6} \\
    2 &=  1 + 1 \\
    1 + 2 + 1 &=  4
\end{align*}
over the row multiplicities. For two of these rows (which we marked in bold in Equations~\ref{eqn:simplify-example} and the relations above), the corresponding values in $v_A$ and $v_B$ equal, so we can remap $\rho$ to be the identity on those rows. Now we have $2$ (trivial) fixed relations and $4$ disjoint row NARs
\begin{align*}
    6 &= 2 + 2 + 2 \\
    4 + 2 &= 6 \\
    2 &= 1 + 1 \\
    1 + 2 + 1 &= 4 \\
    12 &= 12 \\
    6 &= 6
\end{align*}
corresponding to $6$ connected components, $2$ of which are horizontal edges. 
\end{ex}

\subsection{Consequences of Fixed Rows}
\label{sec:column-NAR}

A cospectral pair $\cosp$ (assuming a simplified $\rho$) induces some row NARs using the non-fixed rows, but the fixed rows give us information as well. First, if row $r$ is fixed, then $v_A[r] = v_B[r]$, so
\[
\sum_{c \in A} M[r, c] = \sum_{c \in B} M[r,c]
\]
gives a NAR on the values in the row. We call this the \emph{column NAR (for row $r$)}. This also holds for linear combinations of fixed rows (we skip the proof of this routine Lemma):
\begin{lem}
  \label{lem:fixed-rows-lincomb}
  In a cospectral pair $\cosp$. suppose that two rows $r_1$ and $r_2$ are both fixed. Then let $\{\delta_c \coloneqq \alpha M(N)[r_1, c] + \beta M(N)[r_2, c]\}_{c \in I(N)}$ be a linear combination of the two rows. We must have \[ \sum_{c \in A} \delta_c = \sum_{c \in B} \delta_c.\]
\end{lem}
Thus, it makes sense to talk about the column NAR for e.g. $r_1 + r_2$, where $r_1$ and $r_2$ are different rows. 

\begin{lem}
\label{lem:partition} Suppose $v_1, \ldots, v_N$ are linearly independent. Then let $P = P_1 \cup P_2 \cup \cdots \cup P_k$ be a partition of $[N]$. Suppose we define $v_S$, $S \subset [N]$ to be $\sum_{s \in S} v_s$, then $V_{P_1}, \ldots, V_{P_k}$ are linearly independent as well.
\end{lem}
\begin{proof}
 Suppose $\sum_j \alpha_j V_{P_j} = 0$. Take any $i \in [N]$. It only appears in one of the parts, without loss of generality $P_x$. Since no other $V_{P_{x'}}$ with $x' \neq x$ contains a nontrivial multiple of $v_i$, we must have $\alpha_x = 0$. Repeating the argument for all elements of $[N]$ shows that no nontrivial linear combination of the $V_{P_j}$ can equal $0$, so we are done. 
\end{proof}

We say that a column is \emph{matched} if it is either in both $A$ and $B$ or neither. We will see some analogies between columns being matched and rows being fixed.
\begin{prop}
\label{prop:duality}
Let $\cosp$ be a cospectral pair. Then:
\begin{enumerate}
    \item Row $(n_1, \ldots, n_r)$ is fixed.  Column $(0, \ldots, 0)$ is matched.
    \item At least one column is not matched. At least one row is not fixed. 
\end{enumerate}
\end{prop}
\begin{proof}
Consider the last row $R$ indexed by $(n_1, \ldots, n_r)$. This row has weight $1$ and plays a special role; it contains the largest eigenvalues of the spectra corresponding to the columns. Since this property is stable under addition, we know that the corresponding value must equal in $\overline{v_A}$ and $\overline{v_B}$, so it is fixed. By construction of integral circulant graphs, neither $A$ or $B$ contains the first column, so it is matched. 

For the second part, we already know that $A \neq B$, which implies the column relation is a NAR. It remains to show that not all the rows are fixed, which we prove with a character argument. If all the rows were fixed, we must have $\overline{v_A} = \overline{v_B}$ as vectors. Recall from Section~\ref{sec:prelim} that these vectors are sums of the columns of $\overline{M(N)}$, which are themselves sums over vectors of the form
\[ z_{N,i} \coloneqq [1, \omega^i, \omega^{2i} \ldots, \omega^{(N-1)i}],\]
where $\omega$ is the $N$-th root of unity. These vectors form characters for $Z_N \rightarrow \mathbb{C}$, and so must be linearly independent (see e.g. Artin \cite{artin1998galois}). By Lemma~\ref{lem:partition}, these vectors are also linearly independent, so having $\overline{v_A} = \overline{v_B}$ implies $A = B$, a contradiction.
\end{proof}

As an immediate consequence of the second part of Proposition~\ref{prop:duality},
\begin{cor}
\label{cor:no-NAR-ISAP}
If there is no NAR on $P(N)$, then $N$ satisfies the ISAP.
\end{cor}

In \cite{monius2023many}, M\"onius and So's primary strategy was to show that $P(N)$ for $N = pq^k$, $2 < p < q$ and $N = pqr$, $2 < p < q < r$ are both super sequences. Thus, we can rephrase their strategy as proving that no NARs exist for these $N$ and then using Corollary~\ref{cor:no-NAR-ISAP}. 

The rest of our paper explores further conditions beyond super sequences for when NARs cannot exist, which we then combine with observations about $M(N)$ to eliminate possible counterexamples. We remark that it is not sufficient to \textbf{only} consider the nonexistence of NARs on $P(N)$. In particular, \cite{monius2023many} also proves that  $N=2q^k$ satisfies the ISAP, even though $P(N) = \{1, 1, (q-1), (q-1), \ldots, (q-1)q^{k-1}, (q-1)q^{k-1}\}$ contains NARs (in particular, it contains repeated elements).

\section{General Results}
\label{sec:general}

Given multisets $S_1, \ldots, S_k$, define $\otimes_{i=1}^k S_i$ to be the multiset of $|S_1|\times \cdots \times |S_k|$ numbers that are $k$-wise products coming from picking one element from each set.

\begin{thm}
  \label{thm:general-NAR}
  Let $N$ have the prime decomposition $p_1^{n_1}\cdots p_r^{n_r}$. Suppose that: 
  \begin{enumerate}
  \item for \textbf{all} $i \in [r]$, there exists no NAR on \[P\left(\frac{N}{p_i^{n_i}}\right) = \otimes_{j \neq i} \{1, (p_j - 1), (p_j - 1)p_j, \ldots, (p_j-1)p_j^{n_j-1}\} \pmod{(p_i-1)};\]
  \item there \textbf{exists} an $i \in [r]$ such that there exists no NAR on \[ 
 \otimes_{j \neq i} \{1, p_j, \ldots, p_j^{n_j-1}\} \pmod{p_i},\]
  \end{enumerate}
  Then there exists no NAR on $P(N)$. As a consequence, $N$ satisfies the ISAP.
\end{thm}

\begin{proof}
  Suppose we have a NAR $R$ on $P(N)$. Suppose at least one of the terms corresponds to $\phi(m_1, \ldots, m_r)$ where $m_i = 0$ for some $i$. Since all terms corresponding to $m_i > 1$ contains a factor of $(p_i-1)$, taking mod  $(p_i-1)$ we obtain a NAR on just the terms with $m_i=0$. This is exactly the set given in the first condition, so no such NAR exists.

  Therefore, $R$ must only use the elements where all $m_i \geq 1$. These are precisely terms in the product $\otimes_{j=1} \{(p_j-1), (p_j-1)p_j, \ldots, (p_j-1)p^{n_j-1}\}$. Since all the terms are divisible by $(p_j-i)$, there is a bijection between NARs on this set and NARs on $\otimes_{j=1} \{1, p_j, \ldots, p^{n_j-1}\}$. So we must have a corresponding $R'$ on the latter set.

  Furthermore, for any $i$, if the minimum power of $p_i$ that appears in any of the elements in $R'$ is $m$, then dividing by $p_i^m$ gives another NAR. This means we can further assume that for every $i$, there must exist some element in $R'$ not divisible by $p_i$. This means that taking $\pmod{p_i}$ creates a NAR on just the elements not divisible by $p_i$, which is $\otimes_{j \neq i} \{1, p_j, \ldots, p^{n_j-1}\} \pmod{p_i}$. In other words, if there exists an $i$ such that there is no NAR on this set, we would obtain a contradiction.
\end{proof}

% \begin{cor}
%   \label{cor:modulus-constraint}
% \end{cor}

If we consider the $(n_1, \ldots, n_r)$ as fixed (all $n_i$ roughly having size $m$) and consider random big primes (all $p_i$ roughly having $p$), then the sets that appear in Theorem~\ref{thm:general-NAR} are approximately uniformly random modulo $(p_i-1)$, so each condition is met with probability $\approx 1 - \frac{2^{m^{r-1}}}{p}$. This means as $p \rightarrow \infty$ Theorem~\ref{thm:general-NAR} gives a heuristic proof that our desired property holds for almost all $N$ (of course, if we consider a different distribution then this heuristic does not hold; for starters, the theorem does not even work for any even numbers). We can obtain another such result by noticing that the $p_i$ cannot be too far from one another:

\begin{thm}
\label{thm:far-apart}
Fix prime $p_1$ and natural numbers $n_1, \ldots, n_r$. Then there are only possibly finitely many $N$ of form $N = p_1^{n_1}\cdots p_r^{n_r}$ that do not satisfy the ISAP.
\end{thm}
\begin{proof}
    Let $p_2 > p_1^{n_1} + 1$. We can check that the only terms in $P(N)$ involving $p_1$ and $p_2$ form a super sequence when put in lexographic order sorted by the leading power of $p_2$ and then $p_1$:
\[
1, (p_1-1), p_1(p_1-1), \ldots, p_1^{n_1-1}(p_1-1), (p_2-1), (p_2-1)(p_1-1), (p_2-1)p_1(p_1-1), \dots,
\]
because 
\[
1 + (p_1 - 1) + \cdots + p_1^{n_1-1}(p_1-1) = p_1^{n_1} < p_2 - 1
\]
and addition is otherwise dominated by the power of $p_2$.

Let the sum of all such terms be $C$, and let $p_3 > C + 1$. The same logic shows that all the terms involving only $p_1, p_2, p_3$ can be arranged into a super sequence. Repeating the argument, we can conclude that as long as every $p_i$ is sufficiently big compared to the previous $p_i$, we can put all the elements of $P(N)$ into a super sequence. Then Proposition~\ref{prop:super-sequence} shows that $N$ satisfies the ISAP.
\end{proof}

%Another key takeaway is that it is useful to think of the structure of $P(N)$ as a union of a ``shell'' (for when one of the $n_i$'s equals 0) and an ``interior,'' which have different arithmetic properties. We also observe that dimensional reduction techniques work well for our problem (in particular, the ``interior'' portion has a symmetry that can be exploited). These insights carry over to our second result. 

\begin{lem}
  \label{lem:adjacent-rows}
  In $M(N)$, for each row index $(m_1, \ldots, m_i, \ldots, m_r)$ where $m_i < n_i$, the two rows with indices  $(m_1, \ldots, m_i, \ldots, m_r)$ and $(m_1, \ldots, m_i+1, \ldots, m_r)$ (which differ only in the $i$-th entry), have matching values on all coordinates except the columns $(*, m_i+1, *)$, where
    \begin{align*}
    &\hphantom{\,\,\,\,\,\,\,\,} M(N)[(m_1, \ldots, m_i+1, \ldots, m_r), (m_1', \ldots, m_i+1, \ldots, m_r')] \\ 
    &\hphantom{\,\,\,\,\,\,\,\,} - M(N)[(m_1, \ldots, m_i, \ldots, m_r), (m_1', \ldots, m_i+1, \ldots, m_r')] \\
    &= p_i^{m_i+1} \prod_{j \neq i} M(p_j^{n_j})[m_j, m_j']
    \end{align*}
and the columns $(*, m_i+2, *)$, where
    \begin{align*}
      &\hphantom{\,\,\,\,\,\,\,\,} M(N)[(m_1, \ldots, m_i+1, \ldots, m_r), (m_1', \ldots, m_i+2, \ldots, m_r')] \\
      &\hphantom{\,\,\,\,\,\,\,\,} - M(N)[(m_1, \ldots, m_i, \ldots, m_r), (m_1', \ldots, m_i+2, \ldots, m_r')] \\
      &= -p_i^{m_i+1} \prod_{j \neq i} M(p_j^{n_j})[m_j, m_j'];
    \end{align*}
\end{lem}
\begin{proof}
    Direct observation from $M(N)$.
\end{proof}

\begin{prop}
\label{prop:fix-all-coordinates} Let $N$ have the prime decomposition $p_1^{n_1}\cdots p_r^{n_r}$. Suppose that there exists a cospectral pair $\cosp$ and some $i \in [r]$ such that all $(n_i+1)$ rows indexed $(n_1, \ldots, m_i, \ldots, n_r)$ are fixed. Then there exists a NAR on \[\otimes_{j \neq i} \{1, (p_j-1), (p_j-1)p_j, \cdots, (p_j-1)p_j^{n_j-1}\}.\]
\end{prop}

\begin{proof} 
We prove by contradiction. Suppose no such NAR exists. Consider two of these fixed rows $(n_1, \ldots, n_i, \ldots, n_r)$ and $(n_1, \ldots, n_i-1, \ldots, n_r)$. 
By Lemma~\ref{lem:fixed-rows-lincomb}, $\cosp$ induces a column NAR $X$ on the difference between these rows. By Lemma~\ref{lem:adjacent-rows}, they differ only in columns $(*, n_i, *)$. For each such column $(m_1, \ldots, n_i, \ldots, m_r),$ we have
\begin{align*}
&\hphantom{\,\,\,\,\,\,\,\,} M(N)[(n_1, \ldots, n_i, \ldots, n_r), (m_1, \ldots, n_i, \ldots, m_r)] \\
&\hphantom{\,\,\,\,\,\,\,\,} - M(N)[(n_1, \ldots, n_i-1, \ldots, n_r), (m_1, \ldots, n_i, \ldots, m_r)] \\
&= p_i^{n_i} \prod_{j \neq i} M(p_j^{n_j})[n_j, m_j] \\
&= p_i^{n_i} \prod_{j \neq i} (p_j-1)p_j^{m_j-1}. 
\end{align*}
Equivalently, as we range over all the coordinates except for $i$, the nonzero differences between our two rows form the multiset 
\[
  p_i^{n_i} \cdot \otimes_{j \neq i} \{1, (p_j-1), (p_j-1)p_j, \ldots, (p_j-1)p_j^{n_j-1}\}.
\]
By our assumption, no NAR exists on this set. Thus, none of their entries can be involved nontrivially in $X$, so we can conclude that all columns of the form $(*, n_i, *)$ are matched.

Iterating, consider $(n_1, \ldots, n_i-2, \ldots, n_r)$ and $(n_1, \ldots, n_i-1, \ldots, n_r)$, which are again both fixed, so $\cosp$ induces a column NAR $X$ on their difference. Their coordinates only differ in columns $(*, n_i-1, *)$ and $(*, n_i, *)$. As we just calculated, the multiset of differences on the columns in $(*, n_i - 1, *)$ equals
\[
  p_i^{n_i-1} \cdot \otimes_{j \neq i} \{1, (p_j-1), (p_j-1)p_j, \ldots, (p_j-1)p_j^{n_j-1}\},
\]
which we assumed to be impossible. Since we have already shown that the columns $(*, n_i, *)$ are matched, none of the columns $(*, n_i-1, *)$ can be involved nontrivially in $X$ either, and thus they are also matched. Continuing this logic, all columns in $I(N)$ must be matched, which is a contradiction.
\end{proof}

\begin{thm}
    \label{thm:general-cos}
Let $N$ have the prime decomposition $p_1^{n_1}\cdots p_r^{n_r}$. Suppose that there exists an $i \in [r]$ such that: 
  \begin{enumerate}
  \item there exists no NAR on \[ \{1, (p_i - 1), (p_i - 1)p_i, \ldots, (p_i-1)p_i^{n_i-1}\} \pmod{\gcd(\{p_j-1\}_{j \neq i})};\]
  \item there exists no NAR on \[ 
 \otimes_{j \neq i} \{1, (p_j-1), (p_j-1)p_j,  \ldots, (p_j-1)p_j^{n_j-1}\},\]
  \end{enumerate}
  Then $N$ satisfies the ISAP.
\end{thm}
\begin{proof}
We prove by contradiction and assume a cospectral pair $\cosp$ exists. Let $i$ be the number given by the Theorem assumption. The rows $(n_1, \ldots, m_i, \ldots, n_r)$ where $m_i$ ranges from $0$ to $n_i$ (while all the other coordinates are equal to their maxima $n_j$) have row multiplicities in
\[
J = \{1, (p_i-1), \ldots, (p_i-1)p_i^{n_i-1}\}.
\]The other elements in $P(N)$ are all divisible by at least one $(p_j-1)$ for some $j \neq i$, so they are all divisible by $\gcd(\{p_j-1\}_{j \neq i})$. Taking this modulus, the first condition enforces that none of the elements in $J$ can be nontrivially involved in a NAR. As a result, their rows $(n_1, \ldots, m_i, \ldots, n_r)$ must all be fixed. Proposition~\ref{prop:fix-all-coordinates} then applies, meaning we must have a NAR on \[ 
 \otimes_{j \neq i} \{1, p_j, \ldots, p_j^{n_j-1}\} \pmod{p_i}.\]
As we do not, we have reached a contradiction.
\end{proof}

Unlike Theorem~\ref{thm:general-NAR}, Theorem~\ref{thm:general-cos} does \textbf{not} prove that there exists no NAR on $P(N)$. For example, consider $N = 5^3 19^2$. There is a NAR on $P(N)$ because $(19-1)*19 + (19-1) = (19-1)*(5-1)(5)$. However, as there does not exist a NAR on $\{1, (5-1), (5-1)5, (5-1)5^2\} \pmod{19-1}$ and there does not exist a NAR on $\{1,(19-1), (19-1)*19\}$, Theorem~\ref{thm:general-cos} still applies. This means we really are using additional structural properties of $M(N)$ in addition to number theoretical properties of $N$.

Recall that \cite{monius2023many} already proved the ISAP for $p^k$, $pq^k$, and $pqr$, so the natural next step is $N = p^{n_1}q^{n_2}$. One possible way to use  Theorem~\ref{thm:general-cos} is:
\begin{cor}
  \label{cor:pkqk} Suppose $N = p^{n_1}q^{n_2}$. If there exists no NAR on $\{1, (p - 1), (p - 1)p, \ldots, (p-1)p^{n_1-1}\} \pmod{(q-1)}$, $N$ satisfies the IAP.
\end{cor}

\begin{proof}
It suffices to check the second condition of Theorem~\ref{thm:general-cos}. To see this, note that the multiset $\otimes_{j \neq i} \{1, (p_j-1), (p_j-1)p_j, \cdots, (p_j-1)p_j^{n_j-1}\}$ reduces to a single list, which cannot have a NAR because it is a super sequence.
\end{proof}

In the remaining sections, we give stronger results when one or both of the $n_i$ equal $2$.

\section{$N = p^2q^{n_2}$}
\label{sec:p2qk}

When one of the coefficients equals $2$, we can give a  stronger statement than Corollary~\ref{cor:pkqk}. The proof is similar to that of Theorem~\ref{thm:general-cos}, except at various parts of it we use some alternative tactics.

\begin{thm}
  \label{thm:p2qk} Suppose $N = p^2 q^{n_2}$, where $p,q$ are odd primes and $(q-1) \nmid (p-1)^2(p+1)$, then $N$ satisfies the ISAP.
\end{thm}

\begin{proof}

We prove by contradiction, and assume that there exists some cospectral pair $\cosp$. To start, recall that
\[ M(p^2) = \left[ \begin{array}{ccc|c} 1 & -1 & 0 & \times p(p-1) \\ 1 & (p-1) & -p & \times (p-1) \\ 1 & (p-1) & p(p-1) & \times 1 \end{array} \right] \]
and $M(q^{n_2})$ is some $(n_2+1) \times (n_2+1)$ matrix.

Consider any row or column NAR $X$ on $P(N)$. Since row $(2,n_2)$ is fixed and column $(0, 0)$ is matched by Proposition~\ref{prop:duality}, we can ignore the weight $1$. Thus, the one-sided form of $X$ can be written
\[
  a_0(p-1) + a_1 p(p-1) + a_2 (q-1) + a_3 (p-1)(q-1) + a_4 (q-1)p(p-1) + \cdots = 0.
\]
Suppose $a_0 \neq 0$. Without loss of generality $a_0 = 1$. Looking at this equation mod $(q-1)$ gives 
\[
(p-1) + a_1 p(p-1) = 0 \pmod{(q-1)}.
\]
So this is only possible in $3$ cases:
\begin{itemize}
\item $a_1 = 0$, so $(q-1)|(p-1)$, which violates $(p-1)^2 \neq 0 \pmod{q-1}$. 
\item $a_1 = 1$, so $(q-1) | (p-1) + p(p-1)$, which violates $p^2 - 1 \neq 0 \pmod{q-1}$.
\item $a_1 = -1$: so $(q-1) | p(p-1) - (p-1)$, which again violates $(p-1)^2 \neq 0 \pmod{q-1}$.
\end{itemize}
We have concluded $a_0 = 0$. This means $(p-1)$ cannot be nontrivially involved in any of our row or column NARs, so row $(1, n_2)$ is fixed and column $(1,0)$ is matched.

Suppose row $(0, n_2)$ is also fixed. Then we can apply Proposition~\ref{prop:fix-all-coordinates} to prove that there exists a NAR on
\[
\otimes_{j \neq 1} \{1, (p_j -1), \ldots, (p_j - 1) p_j^{n_j-1}\} = \{1, (q-1), \ldots, (q-1)q^{n_2-1}\},
\]
which is impossible since that is a super sequence. Thus, we know row $(0, n_2)$ is not fixed. This means its weight $p(p-1)$ is involved in a row NAR 
\[
p(p-1) + a_2 (q-1) + a_3 (p-1)(q-1) + a_4 (q-1)p(p-1) + \cdots = 0,
\]
so $(q-1)|p(p-1)$. If $\gcd(q-1, p) = 1$, then we would have $(q-1) | (p-1)$, which is a contradiction since $(q-1) \nmid (p-1)^2.$ Therefore we must have $q-1 = kp$ where $k | (p-1)$. In particular, $k \neq 1$ (else $q$ would be even; we are implicitly using here that $p$ is odd), so $q - 1 \geq 2p$.

Recall that rows $(1, n_2)$ and $(2, n_2)$ are both fixed. Lemma~\ref{lem:adjacent-rows} tells us that their corresponding entries are equal except for entries $(2, *)$, in which case they are negatives of each other. Let $D_+$ be the columns $(i,j)$ where $i < 2$ (alternatively, where row $(1,n_2)$ is positive) and $D_-$ be the columns $(i,j)$ where $i = 2$ (alternatively, where row $(1,n_2)$ is negative). We can observe that the sum of the two rows $r_s$ has nonzero entries only on $D_+$ and the difference $r_d$ has nonzero entries only on $D_-$. Furthermore, the nonzero entries are exactly twice their corresponding entries in row $(2,n_2)$. As an example, if $p=3$, $q=7$, and $n_2=2$, the rows $(1,n_2)$ and $(2,n_2)$, respectively, are
\[
  \begin{array}{c|ccccccccc|c} \hline
  (1,n_2) & 1 & 6 & 42 & 2 & 12 & 84 & -6 & -36 & -252 & (\times 2) \\
(2,n_2) & 1 & 6 & 42 & 2 & 12 & 84 & 6 & 36 & 252 & (\times 1) \\
r_s/2 & 1 & 6 & 42 & 2 & 12 & 84 & 0 & 0 & 0 & \\
r_d/2 & 0 & 0 & 0 & 0 & 0 & 0 & 6 & 36 & 252 & \\
\hline
\end{array}
\]

Applying Lemma~\ref{lem:fixed-rows-lincomb} to $r_s/2$ and $r_d/2$, we obtain that if the one-sided-form of the column NAR on $(2, n_2)$ is
\[
X: \sum_{i \in D_+} a_i M(N)[(2, n_2), i] + \sum_{i \in D_-} a_i M(N)[(2, n_2), i] = 0,
\]
we must also have individually that
\[
X_+: \sum_{i \in D_+} a_i M(N)[(2, n_2), i]  = 0
\]
and
\[
X_-: \sum_{i \in D_-} a_i M(N)[(2, n_2), i] = 0,
\]
so the column NAR for $(2, n_2)$ ``splits'' into a column NAR $X_+$ on $r_s/2$ and also a column NAR $X_-$ on $r_d/2$. Looking at $r_d/2$, if any of the columns in $D_-$ are involved nontrivially in $X_-$, we must have a NAR on 
\[
  p(p-1), p(p-1)(q-1), p(p-1)(q-1)q, \ldots, p(p-1)(q-1)q^{n_2-1}.
\]
These are exactly $p(p-1)$ times the elements in $P(q^k)$, which form a super sequence and thus cannot form NARs.

Thus, all the columns in $D_-$ are matched, and $X = X_+$ only involves columns in $D_+$. These have weights (excluding the first column with weight $1$), 
\begin{align*}
     & \{(q-1), (q-1)q, (q-1)q^2 \ldots, (q-1)q^{n_2-1}\} \\
    \cup & \{(p-1), (p-1)(q-1), (p-1)(q-1)q, (p-1)(q-1)q^2 \ldots, (p-1)(q-1)q^{n_2-1}\}.
\end{align*}

We now use our earlier observation that $q=kp + 1$ for some $k > 1$. Put the above weights in the ``interlaced'' order: 
\[
(p-1), (q-1), (p-1)(q-1), (q-1)q, (p-1)(q-1)q, (q-1)q^2, \ldots
\]
We can check inductively that this is a super sequence (this trick also appears in \cite{monius2023many}):
\begin{enumerate}
    \item $(p-1) < (q-1)$ since $q-1 = kp$, $k > 1$.
    \item $(p-1) + (q-1) < (p-1)(q-1)$ since $(p-1), (q-1)$ are both $\geq 3$. 
    \item If we sum the first $2t+1$ elements (where $t \geq 1$), we obtain
    \begin{align*}
    & (p-1) + (q-1) + \cdots + (q-1)q^{t-1} + (p-1)(q-1)q^{t-1} \\
    &= (p-1)(q^{t}) + (q-1)(q^{t}-1)/(q-1) \\
    & = pq^t -1 < (q-1)q^{t}.
    \end{align*}
    \item If If we sum the first $2t$ elements (where $t \geq 2$) we obtain
    \begin{align*}
    & (p-1) + (q-1) + \cdots + (p-1)(q-1)q^{t-1} + (q-1)q^t \\
    &= pq^t -1 + (q-1)q^{t} \\
    &= (p + q - 1)q^t - 1\\ 
    &< (p-1)(q-1)q^{t}.
    \end{align*}
\end{enumerate}
As a result, the existence of such an $X$ is impossible, and we arrived at a contradiction.
\end{proof}

Compared to Corollary~\ref{cor:pkqk}, this result assumes less because we allow for the case where row $(0, n_2)$ were not fixed. 

\section{$N = p^2q^2$}
\label{sec:p2q2}
% This required some intricate combination of both the arithmetic properties of NARs and the structure of $M(N)$. We use this entire section to give the proof, which will involve several intermediate Lemmas. 

% For brevity, assume that the Lemmas come with assumptions in the context that they were stated (for example, Lemma~\ref{lem:small-values} allows us to assume that $p \geq 11$, so we assuming $p \geq 11$ for the rest of the section).

In this section, we prove that all $N = p^2q^2$ satisfy the ISAP. Without loss of generality, we assume $p < q$ for this entire section.

We first compute $M(N)$. Since the coordinates in $I(N)$ are all single digits, for this section we will suppress the parentheses. For example, we use ``$21$'' as shorthand for $(2,1)$. The table follows, where we omit the unused leftmost column $00$:
\[
\left[  \begin{array}{c|cccccccc|c}
& 01 & 02 & 10 & 11 & 12 & 20 & 21 & 22 & \text{mult.} \\
\hline 
  00 & -1 & 0 & -1 & 1 & 0 & 0 & 0 & 0 & \times \makecell{pq(p-1)\cdot \\ (q-1)} \\
  \hline 
01 & (q-1) & -q & -1 & -(q-1) & q & 0 & 0 & 0  & \times \makecell{p(p-1) \cdot \\ (q-1)} \\
\hline 
  02 & (q-1) & q(q-1) & -1 & -(q-1) & -q(q-1) & 0 & 0 & 0  & \times p(p-1) \\
\hline 
  10 &  -1 & 0 & p-1 & -(p-1) & 0 & -p & p & 0  & \times \makecell{q(p-1)\cdot \\ (q-1)} \\
\hline 
  11 &  (q-1) & -q & p-1 & \makecell{(p-1) \cdot \\ (q-1)} & -(p-1)q & -p & -(q-1)p & pq  & \times \makecell{(p-1)\cdot \\ (q-1)} \\
\hline 
  12 & (q-1) & q(q-1) & p-1 & \makecell{(p-1) \cdot \\ (q-1)} & \makecell{q(p-1) \cdot \\ (q-1)} & -p & -(q-1)p & -pq(q-1)  & \times (p-1) \\
\hline 
   20 & -1 & 0 & (p-1) & -(p-1) & 0 & p(p-1) & -p(p-1) & 0  & \times q(q-1) \\
\hline 
  21 & (q-1) & -q & (p-1) & \makecell{(p-1) \cdot \\ (q-1)} & -(p-1)q & p(p-1) & \makecell{p(p-1) \cdot \\ (q-1)} & -qp(p-1)  & \times (q-1) \\
\hline 
  22 & (q-1) & q(q-1) & (p-1) & \makecell{(p-1) \cdot \\ (q-1)} & \makecell{q(q-1) \cdot\\ (p-1)} & p(p-1) & \makecell{p(p-1) \cdot \\ (q-1)} & \makecell{pq(p-1) \cdot \\ (q-1)}  & \times 1 
 \end{array} \right]
\]

\begin{lem}
\label{lem:small-values}
If $p < 11$, $N = p^2q^2$, and $p < q$, then $N$ satisfies the ISAP.
\end{lem}
\begin{proof}
We can see this Lemma as an example of Theorem~\ref{thm:far-apart} "in practice." For any particular choice of $p$, note that if $q > p^3$, the sequence 
\[
(p-1), p(p-1), (q-1), (p-1)(q-1), p(p-1)(q-1), q(q-1), q(p-1)(q-1), pq(p-1)(q-1) 
\]
is an increasing super sequence, so all the rows must be fixed. This means for us to have a counterexample, for any fixed $p$ we only have to search up to $q = p^3$. We used a computer program to search the values of $p \in \{2, 3, 5, 7\}$ up to their corresponding upper bounds and found no counterexamples.
\end{proof}

A seemingly small but very useful consequence of Lemma~\ref{lem:small-values} is we can assume both $p$ and $q$ are odd. In fact, since we will frequently implicitly use $(p-1) \neq 1$ to argue certain numbers are different, allowing $p=2$ would make many of our arguments false. Our ``precomputation'' also allows us to prove an amusing (and necessary for later!) condition:
\begin{lem}
\label{lem:twin-primes}
If $N = p^2q^2$ and $p$ and $q$ are twin primes, then $N$ satisfies the ISAP.
\end{lem}
\begin{proof}
In this situation, $(q-1) = p+1$ and $q = p+2$. Because row $22$ is fixed, it suffices to consider the weights of the other $8$ rows. These have weights
\[
(q-1), q(q-1), (p-1), (q-1)(p-1), p(p-1), q(p-1)(q-1), p(p-1)(q-1), pq(p-1)(q-1). 
\]
Substituting, we get that they have the weights
\[
(p+1), (p+2)(p+1), (p-1), (p+1)(p-1), p(p-1), (p+2)(p-1)(p+1), p(p-1)(p+1), p(p+2)(p-1)(p+1).
\]
Note that the first two weights (corresponding to rows $21$ and $20$ respectively) are the only ones not immediately divisible by $(p-1)$, so their combined involvement in any (disjoint) row NAR must be divisible by $(p-1)$. When $p > 3$ we cannot have $(p-1) | p + 1$. When $p > 5$ we cannot have $(p-1) | (p+1)^2$. When $p > 7$ we cannot have $(p-1) | (p+2)(p+1)$, when $p > 9$ we cannot have $(p-1) | [(p+2)(p+1) + (p+1)]$, so since we know $p \geq 11$ by Lemma~\ref{lem:small-values}, these 2 elements cannot be involved in a NAR. Their corresponding rows $21$ and $20$ thus must be fixed.

As in the proof of Theorem~\ref{thm:general-cos}, if the rows $22, 21, 20$ are all fixed and $\{1, (p-1), (p-1)p\}$ has no NAR (which is true since $p > 2$), $N$ satisfies the ISAP.
\end{proof}

\begin{lem}
\label{lem:first-row}
If $N = p^2q^2$ with $p < q$, then the values $pq(p-1)(q-1)$ and $q(q-1)(p-1)$ cannot be involved nontrivially in any NARs on $P(N)$. As a consequence, if $A$ and $B$ are cospectral with this $N$, rows $00$ and $10$ must be fixed and columns $22$ and $12$ must be matched.
\end{lem}
\begin{proof}
The sum of $I(N)$ is $p^2q^2$. By Lemma~\ref{lem:small-values}, $p,q$ are odd. By elemntary algebra, this implies $2(p-1)(q-1) > pq$. Thus, $pq(p-1)(q-1) > p^2q^2/2$, which implies that this single value is greater than the sum of all other values in $I(N)$. This means if $pq(p-1)(q-1)$ appears on one side of an additive relation on $I(N)$, it must appear on the other as well, so it cannot be involved nontrivally in any NARs on $I(N)$. Row $00$ and column $22$ correspond to weight $pq(p-1)(q-1)$, so they must be fixed and matched respectively.

We can now remove the value $pq(p-1)(q-1)$ from our consideration and compare the second largest element $q(q-1)(p-1)$ with the other values in $P(N)$. consider 
\begin{align*}
& q(q-1)(p-1) - q(q-1) - (q-1) - (p-1) - (p-1)(q-1) - p(p-1) - p(p-1)(q-1) \\
&= q(q-1)(p-2) - (qp-1) - qp(p-1) \\
&= q(q-1)(p-2) + 1 - qp^2 \\
&\geq q[(p+3)(p-2) - p^2] + 1,
\end{align*}
which is true when $p \geq 7$. The inequality in the last line used $q \geq p+4$, which is true since $p$ and $q$ cannot be twin primes by Lemma~\ref{lem:twin-primes}. This means $q(q-1)(p-1)$ cannot be involved nontrivially in any NAR on $P(N)$ either, so the corresponding row $10$ and column $12$ must be fixed and matched respectively.
\end{proof}

We are now ready to prove our main result.

\begin{thm}
\label{thm:p2q2}
Let $N = p^2q^2$ with $p < q$. Then $N$ satisfies the ISAP.
\end{thm}
\begin{proof}
We prove by contradiction. Suppose not, then there exist some cospectral $A$ and $B$. We know that rows $22$ and $00$ are fixed. By Lemma~\ref{lem:fixed-rows-lincomb}, we know that there exists a column NAR for the difference between these two rows
\[
\begin{bmatrix} 01 & 02 & 10 & 11 & 20 & 21 \\ 
\hline
q & q(q-1) & 0 & q(p-1) & p(p-1) & p(p-1)(q-1)  \end{bmatrix}.
\]
In the expression above, we omitted columns $00$, $12$, $22$ by Lemma~\ref{lem:first-row} because we only care about nontrivial contributions to this NAR. Taken mod $q$, the entries become 
\[
\begin{bmatrix} 0 & 0 & 0 & 0 & p(p-1) & -p(p-1) \end{bmatrix} \pmod{q}.
\]
Because $q > p$, $p(p-1) \neq 0 \pmod{q}$. This means the last two columns $20$ and $21$ must appear together (in exactly one of $A$ or $B$) or be matched. Looking at row $10$ individually, which has entries
\[
\begin{bmatrix} 01 & 02 & 10 & 11 & 20 & 21 \\ 
\hline
-1 & 0 & (p-1) & -(p-1) & -p & p  \end{bmatrix},
\]
this means the columns $20$ and $21$ contributes a net $0$ to the column sums for $A$ or $B$ for this row. Therefore, column $01$ must be matched, since there is no way (implicitly using $p > 3$) to cancel out $-1$ with just a single copy of $(p-1)$ and $-(p-1)$. Thus,  columns $10$ and $11$ must appear together (in exactly one of $A$ or $B$) or be matched. 

We can restate our observations about rows $20, 21, 10, 11$ as saying that we can write down NAR on $3$ columns $\{02, 10+11, 20+21\}$ where, for example, column $10+11$ corresponds to the sum of the two corresponding column weights. We obtain 
\[
\begin{bmatrix}
02 & 10+11 & 20+21 \\
\hline
q(q-1) & q(p-1) & qp(p-1)
\end{bmatrix},
\]
which we can divide out by $q$ to obtain $(q-1), (p-1), p(p-1)$. If $(p-1) + p(p-1) = (q-1)$, then $q=p^2$, a contradiction. This means this NAR must be $(q-1) + (p-1) = p(p-1)$, which means, without loss of generality, $A$ has columns  $\{02, 10, 11\}$ and $B$ has columns $\{20, 21\}$ (and some subset of the other columns appear in both). This means row $01$ must be not fixed, because otherwise we would obtain the equation
\[
(-q) + (-1) + (-(q-1)) = 0,
\]
which is impossible. 

We also know that $I(N)$ must induce at least one (disjoint) row NAR, where rows $00, 22, 10$ are not involved because they are fixed. Take any such NAR $X$ and write it in the one-sided form
\[
a_{21}(q-1) + a_{20}q(q-1) + a_{12}(p-1) + a_{11}(p-1)(q-1) + a_{02}p(p-1) + a_{01}p(p-1)(q-1) = 0,
\]
where each $a_{i}$ is in $\{-1, 0, 1\}$ depending on if it appears on the left-hand-side, neither, or the right-hand-side of $X$ respectively. Since the terms on $a_{11}$ and $a_{01}$ are divisible by $(p-1)(q-1)$, we must have (switching signs for the $a_i$'s if necessary)
\[
a_{12}(p-1) + a_{02}p(p-1) + a_{21} (q-1) + a_{20} q(q-1) = k(p-1)(q-1),
\]
for some $k \in \ZZ, k \geq 0$. The maximum possible sum of the $4$ terms on the left is $(p-1) + p(p-1) + (q-1) + q(q-1) = (p+1)(p-1) + (q+1)(q-1)$. So we know that $k$ is at most 
\begin{align*}
\frac{(p+1)(p-1) + (q+1)(q-1)}{(p-1)(q-1)} & = \frac{p+1}{q-1} + \frac{q+1}{p-1} \\
& = \frac{p+1}{(p-1)^2} + \frac{(p-1)^2 + 2}{p-1} \\
& = (p-1) + \frac{p+1 + 2(p-1)}{(p-1)^2} \\
& < p.
\end{align*}
Since $k$ is an integer, $k \leq (p-1)$. Therefore, because 
\[k(p-1)(q-1) + a_{11} (p-1)(q-1) + a_{01} p(p-1)(q-1) = 0,\] we can divide out to get
\[
k + a_{11} + a_{01}p = 0. 
\]

Because we showed earlier that row $01$ is not fixed, we can assume we picked an $X$ where $a_{01} \neq 0$ (there should be exactly two such choices, from the construction of the row NARs). Since $(p-1) \geq k > 0$, this means we must have $a_{01} = -1$, $a_{11} = 1$, and $k = (p-1)$. Therefore,
\begin{align*}
(p-1)^2(q-1) &= a_{12}(p-1) + a_{02}p(p-1) + a_{21} (q-1) + a_{20} q(q-1) \\
(p-1)^4 &= a_{12}(p-1) + a_{02}p(p-1) + a_{21} (p-1)^2 + a_{20} q(p-1)^2 \\
(p-1)^3 &= a_{12} + a_{02}p + a_{21} (p-1) + a_{20} q(p-1).
\end{align*}
It is clear that we need $a_{20} = 1$, else the right-hand-side is not big enough. Thus, we have concluded that if $01$ is on one side of a row NAR, $11$ and $20$ must both be on the other side.

%This gives
% \begin{align*}
% (p-1)^3 - (1+(p-1)^2)(p-1) &= a_{12} + a_{02}p + a_{21} (p-1) \\
% -(p-1) &= a_{12} + a_{02}p + a_{21}(p-1).
% \end{align*}
% There are two ways to proceed, depending on if $a_{21}$ equals $-1$ or $0$. We conclude that a row NAR with $a_{01} \neq 0$ takes one of two possible forms:
% \begin{enumerate}
%     \item $01, 21$ on one side, $11, 20$ on the other;
%     \item $01, 12$ on one side, $11, 20, 02$ on the other.  
% \end{enumerate}

Finally, let $C$ be the set of columns that appear in both $A$ and $B$. This means $v_A = v_{A'} + v_C$ and $v_B = v_{B'} + v_C$, where $A'= \{02, 10, 11\}$ and $B'= \{20, 21\}$. We can compute 
\[
\begin{array}{c|c|c}
 & v_A' & v_B' \\
 \hline
00 & 0 & 0 \\
01 & -2(q-1) & 0 \\
02 & q(q-2) & 0 \\
10 & 0 & 0 \\
11 & (p-2)q & -pq \\
12 & q(p+q-2) & -pq \\
20 & 0 & 0 \\
21 & q(p-2) & pq(p-1) \\
22 & pq(p-1) & pq(p-1). 
\end{array}
\]
Recall that we just proved that there must be exactly $2$ row NARs with $a_{01} \neq 1$ and $a_{11}, a_{20}$ having the opposite sign as $a_{01}$. By construction of the row NARs (recall that they encode which entries in $v_A$ equal which entries in $v_B$ via $\rho$), we obtain that $v_A[01] = v_B[11] = v_B[20]$ and $v_B[01] = v_A[11] = v_A[20]$. This implies  
\[
0 = v_B[11] - v_B[20] = (v_{B'}[11] + v_C[11]) - (v_{B'}[20] + v_C[20]),
\]
so we know that $(v_{B'}[20] - v_{B'}[11]) = (v_C[11] - v_C[20])$, which must also be equal to $(v_{A'}[20] - v_{A'}[11])$ by a symmetric argument. However, $(v_{A'}[20] - v_{A'}[11] = -(p-2)(q)$ and $(v_{B'}[20] - v_{B'}[11]) = pq$, which gives a contradiction. 
\end{proof}

\section{Conclusion}
\label{sec:conclusion}

The main generalizable contributions of this paper are exploiting the multiplicative structure of $M(N)$ to compactly represent spectra of ICGs and identifying row and column NARs induced by a cospectral pair; these techniques can be strengthened and reused in future work on ICGs. Both the $N = p^2q^2$ and $N=p^2q^{n_2}$ results involved nontrivial amounts of ad-hoc tinkering that was hard for us to generalize. 

As an example of the underlying complexity, one of our main strategies was extracting information (in form of column NARs) when we can prove or assume that rows are fixed. However, even in the ``small'' $N=p^2q^2$ case, we have found plausible NARs involving every weight outside of $\{1, q(p-1)(q-1), pq(p-1)(q-1)\}$ (although not simultaneously). This implies it is a priori impossible to assume any rows are fixed outside of rows $\{00, 10, 22\}$. This situation remains an obstacle for bigger $N$.

\section*{Acknowledgments}

We thank Wasin So for introducing us to this problem and valuable discussion.

\bibliographystyle{abbrv}
\bibliography{bibliography}

%\appendix

\end{document}